\documentclass[11pt,a4paper,fleqn]{article}
\usepackage[a4paper]{geometry}
\usepackage{graphicx}
\usepackage{amsmath,amssymb,latexsym,graphics,epsfig}
\usepackage{hyperref}
\usepackage{color}
\usepackage{amsthm}

\usepackage{setspace}

\usepackage[space]{cite}
\usepackage{fullpage}

\newtheorem{theorem}{\bf Theorem}[section]

\newtheorem{lemma}[theorem]{\bf Lemma}

\newtheorem{claim}[theorem]{\bf Claim}
\newtheorem{remarkit}[theorem]{\bf Remark}

\newtheorem{definitiona}[theorem]{\bf Definition}

\newenvironment{definition}{%
  \begin{definitiona}
    \rm 
  }{\end{definitiona}}
\newenvironment{remark}{%
  \begin{remarkit}
    \rm 
  }{\end{remarkit}}

\usepackage[mathlines]{lineno}
\usepackage{etoolbox} 

\newcommand*\linenomathpatch[1]{%
   \expandafter\pretocmd\csname #1\endcsname {\linenomath}{}{}%
   \expandafter\pretocmd\csname #1*\endcsname{\linenomath}{}{}%
   \expandafter\apptocmd\csname end#1\endcsname {\endlinenomath}{}{}%
   \expandafter\apptocmd\csname end#1*\endcsname{\endlinenomath}{}{}%
 }
\newcommand*\linenomathpatchAMS[1]{%
    \expandafter\pretocmd\csname #1\endcsname {\linenomathAMS}{}{}%
    \expandafter\pretocmd\csname #1*\endcsname{\linenomathAMS}{}{}%
    \expandafter\apptocmd\csname end#1\endcsname {\endlinenomath}{}{}%
    \expandafter\apptocmd\csname end#1*\endcsname{\endlinenomath}{}{}%
}

\expandafter\ifx\linenomath\linenomathWithnumbers
\let\linenomathAMS\linenomathWithnumbers
\patchcmd\linenomathAMS{\advance\postdisplaypenalty\linenopenalty}{}{}{}
\else
\let\linenomathAMS\linenomathNonumbers
\fi

\linenomathpatchAMS{gather}
\linenomathpatchAMS{multline}
\linenomathpatchAMS{align}
\linenomathpatchAMS{alignat}
\linenomathpatchAMS{flalign}
\linenomathpatch{numcases}
\linenomathpatch{figure}

\begin{document}

\title{\Large The Multicolor Size-Ramsey Number of Bipartite Long
  Subdivisions\footnote{Research partially supported by FAPESP
    (2023/03167-5) and CAPES (Finance Code 001)}}

\date{}	

\author{Ramin Javadi\thanks{Department of Mathematical Sciences,
    Isfahan University of Technology, Isfahan, 84156-83111,
    Iran. E-mail: rjavadi@iut.ac.ir}
  \and
  Yoshiharu Kohayakawa\thanks{Instituto de Matem\'atica e
    Estat\'istica, Universidade de S\~ao Paulo, 05508-090 S\~ao
    Paulo, Brazil.  E-mail: yoshi@ime.usp.br.  Partially supported
    by CNPq (406248/2021-4, 407970/2023-1, 315258/2023-3)}
  \and Meysam Miralaei\thanks{Instituto de Matem\'atica e
    Estat\'istica, Universidade de S\~ao Paulo, 05508-090 S\~ao Paulo,
    Brazil.  E-mail: m.miralaei@ime.usp.br.  Supported by FAPESP
    (2023/04895-4)}}

\maketitle

\onehalfspace

\begin{abstract}
  For a positive integer $r$, the $r$-color size-Ramsey
  number~$\widehat{R}_r(H)$ of a graph $H$ is the minimum number of
  edges in a graph $G$ such that every $r$-edge coloring of $G$
  contains a monochromatic copy of $H$.  For a graph~$H$ and a
  function $\sigma:E(H)\to \mathbb{N}$, the \emph{subdivision}
  $H^\sigma$ is obtained by replacing every $e \in E(H)$ with a path
  of length $\sigma(e)$.  In~\cite{javadi25:_induced_long} it is shown
  that for all integers $r,\, D\geq 2 $, there exists a constant
  $c=c(r, D)$ such that for every graph $ H $ with maximum degree $D$
  if $H^{\sigma}$ is a subdivision of~$H$ in which
  $\sigma(e) > c \log n $ for every $e \in E(H)$, where
  $n=|V(H^\sigma)|$, then
  $ \widehat{R}_r(H^\sigma) = O\big(2^{34r}
        r^6 \log^5(r) D^5\log D\big)n.  $ We improve upon this result in the case that~$H^{\sigma}$
  is a bipartite graph and the number of colors~$r$ is large using a
  significantly different argument, obtaining the bound
  $ \widehat{R}_r(H^{\sigma}) \leq r^{400D \log D} \, n $.
\end{abstract}
	

\section{Introduction}	
For given graphs $H$ and $G$, and a positive integer $r$, we say that $G$ is \emph{Ramsey for} $H$, denoted by $G \longrightarrow (H)_r$, if for every $r$-edge coloring of $G$ with $r$ colors, there exists a monochromatic subgraph of $G$ isomorphic to $H$.
Ramsey's pioneering work~\cite{Ramsey} established that for any graph
$H$ and any positive integer $r$, there exists $N\in \mathbb{N}$  such
that $K_N \longrightarrow (H)_r$. The smallest such~$N$ is called the
\emph{$r$-multicolor Ramsey number} of $H$ and is denoted by $R_r(H)$.

While this perspective focuses on minimizing the number of vertices,
an alternative viewpoint—initiated by Erd\H{o}s, Faudree, Rousseau,
and Schelp~\cite{S.R.Erd}—seeks to minimize the number of edges. This
leads to the notion of \emph{size-Ramsey numbers}. For a graph $H$ and
integer $r \geq 2$, the \emph{$r$-multicolor size-Ramsey number}
of~$H$ is 
\[
\widehat{R}_r(H) = \min \{ |E(G)| : G \longrightarrow (H)_r \}.
\]

One of the central problems in this area is to determine for which
families of graphs the size-Ramsey number grows linearly with the
number of vertices. Erd\H{o}s~\cite{Erd_1} asked this question in the
case of paths and this question was addressed by Beck~\cite{Beck_1},
who proved that $\widehat{R}_2(P_n) < 900n$.  The linearity of size-Ramsey
numbers has since been established for various graph families:
cycles~\cite{Sudakov-cycle, Haxell-cycle, JKhOP, JM}, bounded-degree
trees~\cite{Friedman, Haxell-tree, Xin}, powers of paths and
cycles~\cite{Power_path, jie}, bounded powers of bounded-degree trees
(equivalently graphs of bounded treewidth and maximum
degree)~\cite{Power_tree, treewidth, jiang13:_degree_ramsey}, and logarithmic subdivisions of bounded-degree graphs~\cite{Rolling}. On the negative side, R\"{o}dl and Szemer\'{e}di~\cite{Rodl} constructed graphs on $n$ vertices with maximum degree $3$ and size-Ramsey number at least  $n \log^cn$, for a small constant $c>0$. This bound has  been
improved to $ cne^{c\sqrt{\log n}} $ for some $ c > 0 $ by Tikhomirov~\cite{Tikh}. Thus
the linearity of size-Ramsey numbers does not extend universally to bounded degree graphs. 

Given a graph $H$ and a function $\sigma:E(H)\to \mathbb{N}$, the \emph{subdivision} $H^\sigma$ is obtained by replacing each edge $e \in E(H)$ with a path of length $\sigma(e)$.
Answering a conjecture of 
Pak~\cite{Pak}, Dragani\'c, Krivelevich, and Nenadov~\cite{Rolling} proved the following theorem.

\begin{theorem}
  {\rm~\cite{Rolling}}
  \label{conj:pak}
  For every $ r,\, D\in \mathbb{N}  $ there exist  $ C,\,L > 0 $ such that if $ H $
  is a graph with $ \Delta(H) \leq D $ and $ \sigma(e)\ge L\log(|V(H^{\sigma})|) $ for all $ e \in E(H) $, then
  $\widehat{R}_r(H^{\sigma})\leq C|V(H^{\sigma})| $.
\end{theorem}

The proof of Theorem~\ref{conj:pak} is based on Szemer\'edi's regularity lemma and makes no attempt to optimize the constant~$C$. 

In a recent paper~\cite{javadi25:_induced_long} we reprove
Theorem~\ref{conj:pak} with the goal of obtaining reasonable 
explicit bounds for the constants, in particular avoiding the use of
regularity methods. Specifically, our main result for general
subdivisions in~\cite{javadi25:_induced_long} is as follows.

\begin{theorem}{\rm~\cite{javadi25:_induced_long}}
  \label{thm:odd_subd}
  Let $r,\, D \geq 2$ be integers, and let $H^\sigma$ be a subdivision
  of a graph~$H$ with maximum degree~$D$ with
  $ \sigma(e) > {2\log_2^2r\log_{D}n+Cr^4 2^{2r}\log^2 D}$
  for every $e \in E(H)$, where $n=|V(H^\sigma)|$ and~$C$ is a
  suitably large absolute constant. Then
  \begin{equation}
    \label{eq:1}
    \widehat{R}_r(H^\sigma) = O\big(2^{34r}
        r^6 \log^5(r) D^5\log D\big)n.
  \end{equation}
\end{theorem}

In~\cite{javadi25:_induced_long} we also prove that if we further
restrict to the case that $\sigma(e)$ is even for every $e \in E(H)$,
then Theorem~\ref{thm:odd_subd} can be significantly improved.

\begin{theorem}{\rm~\cite{javadi25:_induced_long}}
  \label{thm:even_subd}
  Let $r, \, D \geq 2$ be integers, and let $H^\sigma$ be a
  subdivision of a graph~$H$ with maximum degree~$D$ with $\sigma(e)$
  even and larger than ${2\log_2^2r\log_{D}n+Cr^4 2^{2r}\log^2 D}$ for
  every $e \in E(H)$, where $n=|V(H^\sigma)|$ and~$C$ is a suitably
  large absolute constant.  Then
  \begin{equation}
    \label{eq:2}
    \widehat{R}_r(H^\sigma) = O\big(r^{34}
        D^5\log D\big)n.
  \end{equation}
\end{theorem}

Clearly, the subdivisions~$H^\sigma$ considered in
Theorem~\ref{thm:even_subd} are bipartite graphs.  In this paper we
prove that an additional improvement of Theorem~\ref{thm:odd_subd} is
possible when the subdivided graph is bipartite, \textit{without}
assuming that every~$\sigma(e)$ is even.

\begin{theorem}
  \label{thm:main}
  Let $r, \, D\geq 2$ be integers, and let~$H$ be a graph with maximum
  degree~$D$.  Suppose~$H^\sigma$ is a subdivision of~$H$
  with~$\sigma(e)\geq 2\log_{D-1} n$ for every $e \in E(H)$, where
  $n=|V(H^\sigma)|$.  Furthermore, suppose~$H^\sigma$ is bipartite.
  Then
  \begin{equation}
    \label{eq:3}
    \widehat{R}_r(H^\sigma) \leq r^{\,400 D\log D} \, n.
  \end{equation}
\end{theorem}

In general, the bound in~\eqref{eq:3} is worse than the bound
in~\eqref{eq:2}, but of course~\eqref{eq:3} applies to more general
subdivisions~$H^\sigma$.  Furthermore, the method of proof employed
here is entirely different from the one
in~\cite{javadi25:_induced_long}. 

\subsection*{Conventions and notation}

For a graph $ G $, we write $ V(G) $, $ E(G) $ and $ e(G) $ for the vertex set, the edge set and the number of edges of $ G $, respectively.  
For $ v \in V(G) $, by $ N_{G}(v) $ we mean the set of all neighbors of $ v $ in $G$.
For a subset $ X\subseteq V(G) $, we define the neighborhood of $X$ as~$ N_G(X)=\bigcup_{x\in X}N_G(x) $.
In a rooted tree $T$, the maximum distance of a vertex from the root is the height of $ T $. If a tree has only one vertex (the root), its height is zero.
Let $ A,\, B \subset V(G) $ with $A\cap B=\emptyset$. Then $ E_G(A,B)= \{ xy\in E(G) : x\in A, \, y\in B\} $
is the set of edges connecting a vertex of $ A $ to a vertex of $ B $. Also, $e_G(A,B)=|E_G(A,B)|$. A bipartite graph $ G $ with a bipartition $ (V_1, V_2) $ is denoted by $ G=G(V_1, V_2) $.

Given functions $f(n), \, g(n)$ and $h(n)$, we write $\Omega(h(n))=f(n)=O(g(n))$ if there exist  absolute constants $ C_1,\,  C_2>0 $ such that $C_1h(n)\leq f(n)\leq C_2g(n)$ for  all sufficiently large~$n$.

All logarithms are taken to base 
$ e $, unless stated otherwise. Throughout the paper, we omit floor and ceiling symbols whenever they are not essential.
 
\section{Tools}
In this section we collect several auxiliary results and notions that will be used throughout the paper.

We begin by recalling the model of a random bipartite graph.
 The \emph{binomial random bipartite graph} $G(N,N,p)$ is defined as the probability distribution over the family of bipartite graphs $G=G(V_1,V_2)$ with $|V_1|=|V_2|=N$, in which each potential edge $\{i,j\}$ with $i\in V_1$ and $j\in V_2$ appears independently with probability $p$. We say that
a property $A_N$ of $G(N, N, p)$ holds \emph{asymptotically almost surely} (a.a.s.) if the probability that $A_N$ holds tends to $1$ as $N\to\infty$.

It was shown in~\cite{JM} that there exist explicit bipartite graphs whose local edge densities mimic those expected in random bipartite graphs. The following lemma that is Lemma 2.2 in~\cite{JM}, provides a precise quantitative formulation of this phenomenon.
\begin{lemma}{\rm~\cite{JM}}\label{quasi}
	Let $c_1$ be a positive integer and $c_2$, $c_3$, $\varepsilon$, $\delta$ be positive numbers such that $c_3\leq c_1$, $0<\epsilon\leq 1/2$ and $3/2\geq \delta >\sqrt{6{\log({c_1e}/{c_3})/(c_2c_3)}} $. Then there exists $n_0=n_0(c_1,c_2,c_3,\epsilon, \delta)$ for which the following holds for every $n\geq n_0$. 
	
	Let $ N=c_1n $ and $ p={c_2/n} $. 
	There exists a bipartite graph $ G=G(V_1,V_2) $, with $ |V_i|=N $, $ i\in\{1,2\} $, such that
	\begin{enumerate}
		\item \label{quasi 1} $ (1-n^{\varepsilon-1/2}) pN^2 \leq |E(G)|\leq (1+n^{\varepsilon-1/2}) pN^2 $,
		\item \label{quasi 2} For every two subsets $ U\subseteq V_1 $ and $ W\subseteq V_2 $ with $ |U|=u$ and $ |W|=w$, where $u, \,w\geq c_3n $, we have
		\[
		\big| e_G(U,W) -puw\big|\leq \delta puw.
		\]
	\end{enumerate} 
\end{lemma}
Next, we introduce a convenient framework to describe bipartite graphs with bounded part sizes and degree constraints.

For positive integers $n_1$, $n_2$, $D_1$ and $D_2$ an $(n_1,n_2;D_1,D_2)$-bipartite graph $G$ is a graph with a bipartition $(X_1,X_2)$, such that  $|X_i|\leq n_i$ and the maximum degree of vertices in $X_i$ is at most $D_i$ for each $i\in\{1,2\}$.  When $n_1=n_2=n$ and $D_1=D_2=D$, we simply write $(n,D)$-bipartite graph.
\begin{definition}
Let $n$ be a positive integer and $D>0$. We say that a bipartite graph $G=G(V_1,V_2)$ is $(n,D)$-expanding if for each $i\in \{1,2\}$ and every subset $X\subseteq V_i$ with $|X|\leq n$, we have $|N_G(X)|\geq D |X|$. 
\end{definition}
The central tool in our work is the classical tree-embedding lemma of Friedman and Pippenger~\cite{Friedman}. This fundamental result provides a general expansion criterion that guarantees the existence of embeddings of bounded degree trees into large graphs. Intuitively, under sufficient expansion, one can embed a tree greedily by starting at the root and at each step there will be enough fresh neighbors available to add a pendent vertex. A bipartite refinement of this result was later obtained by Haxell and Kohayakawa~\cite{Haxell-tree}. 

More recently, Dragani\'c, Krivelevich, and Nenadov~\cite{Rolling} introduced the notion of \emph{good embeddings}, which provides a flexible framework for handling iterative embedding arguments. For our purposes we need the bipartite analogue of this notion.

\begin{definition}\label{def:good}{}
	Let $n$ and $D$ be positive integers. Suppose that $G = G(V_1, V_2)$ and $F$ are bipartite graphs and let $\phi: F \hookrightarrow G$ be an embedding of $F$ into $G$. We say that $\phi$ is \emph{$(n,D)$-good} if for each $i \in \{1,2\}$ and every subset $X \subseteq V_i$ with $|X| \le n$, we have 
	\begin{equation}\label{good-embedding}
	\big| N_G(X) \setminus \phi(F) \big| \ \ge \ \sum_{v \in X} \big( D - \deg_{F}(\phi^{-1}(v)) \big) \;+\; \big| \phi(F) \cap X \big|.
	\end{equation}
	Here, for vertices $v \in V(G)$ not used by $\phi$ in the embedding, we set $\deg_{F}(\phi^{-1}(v)) := 0$.
\end{definition}
Definition \ref{def:good} coincides with the one used by Friedman and
Pippenger~\cite{Friedman} and Haxell and Kohayakawa~\cite{Haxell-tree} up to the additional term on the right hand side of~\eqref{good-embedding}, which accounts for vertices already occupied by the embedding.

The next theorem shows that under suitable expansion assumptions, the property of being a good embedding is preserved when extending the embedded graph by attaching pendent vertices. This is the bipartite counterpart of~\cite[Theorem~2.3]{Rolling}. As the argument is essentially the same, we only sketch it in Appendix~\ref{app}.
\begin{theorem} \label{thm:good_adding}
Let $n$ and $D$ be  positive integers. Consider bipartite graphs  $F$ and $G$ and suppose that $\phi: F\hookrightarrow G$ is a $(2n,D)$-good embedding. 
Assume further that $F$ is $(n,D)$-bipartite and $G$ is   $(2n,D+2)$-expanding. If $F'$ is an $(n,D)$-bipartite graph which is obtained from $F$ by successively adding new pendent vertices, then there exists a $(2n,D)$-good embedding $\phi': F'\hookrightarrow G$.
\end{theorem}
 
In a complementary direction, the following lemma shows that the goodness property is also preserved when successively removing degree (at most) one vertices from the embedded graph. This is the bipartite analogue of~\cite[Lemma~2.4]{Rolling} and we include a proof in Appendix~\ref{app2} for completeness.
\begin{lemma}\label{lem:good_removing}
Let $n$ and $D$ be  positive integers. Consider an $(n,D)$-bipartite graph $F$ and suppose that $\phi: F \hookrightarrow G$ is an $(n,D)$-good embedding. Assume that $F'$ is obtained from $F$ by successively deleting vertices of degree at most one. Then the restriction of $\phi$ to $F'$ is also $(n,D)$-good.
\end{lemma}
We next recall another useful structural concept introduced in~\cite{JM}, which allows us to pass from certain dense bipartite configurations to bipartite graphs with strong expansion.
\begin{definition}{\rm~\cite{JM}}
Let $G=G(V_1,V_2)$ be a bipartite graph with $|V_1|=|V_2|=N$ and let $\alpha\in(0,1)$ be a real number. We say that $G$ is \emph{$\alpha$-joined} if for every pair of subsets $A\subseteq V_1$ and $B\subseteq V_2$ with $|A|, \,|B|\geq \alpha N$, we have $e(A,B)\neq0$.
\end{definition}
The next lemma shows that every $\alpha$-joined bipartite graph contains a large induced subgraph with strong expansion properties. Moreover, this subgraph also supports a good embedding of a suitable null graph.
\begin{lemma} \label{lem:expender_in_join} 
Let $G=G(V_1,V_2)$ be an $\alpha$-joined bipartite graph with $0<\alpha<1/5$ and $|V_1|=|V_2|=N$. Then there exists an induced subgraph $G'=G'(V'_1,V'_2)$ of $G$ such that for each $i\in \{1,2\}$, 
	\begin{itemize}
		\item $|V'_i|\geq (1-\alpha)N$,
		\item for every $X\subset V'_i$, with $|X|\leq \alpha N$, we have $|N_{G'}(X)|> (1-4\alpha)/{(2\alpha)} |X|$, and
		\item for every $X\subset V'_i$, with $|X|>  \alpha N$, we have $|N_{G'}(X)|> (1-2\alpha)N$.
	\end{itemize}
	In particular, $G'$ is $(6\alpha N, (1-2\alpha)/(6\alpha))$-expanding. 
	Moreover, there exists an embedding $\phi: F\hookrightarrow G'$, where $F$ is a null bipartite graph with $\alpha N$ vertices in each part, such that $\phi$ is a
	$(6\alpha N, {(1-4\alpha)}/{(6\alpha)})$-good embedding.
\end{lemma}
\begin{proof}

	For each $i\in \{1,2\}$, let $Y_i$ be an arbitrary subset of $V_i$ with $|Y_i|=2\alpha N$. If there exists a set $X\subset V_i$, with $|X|\leq \alpha N$ and $|N_G(X)\setminus (Y_1\cup Y_2)|\leq {(1-4\alpha)}/{(2\alpha)}|X|$, then we remove such an $X$ from $G$ and we iterate this process. Let $X_1$ and $X_2$ denote the removed vertices from $V_1$ and $V_2$, respectively.
	
	 We claim that $|X_1|,\, |X_2|\leq \alpha N$. Suppose otherwise, and assume w.l.o.g.\ that~$|X_2|>\alpha N$. Consider the moment when $X_2$ first exceeds $\alpha N$. Then $|X_1|\leq \alpha N$ and $\alpha N< |X_2|\leq 2\alpha N $. Therefore,
	\[
	|N_G(X_2)|\leq \dfrac{1-4\alpha}{2\alpha} |X_2|+ |Y_1|+|X_1|\leq 
	\dfrac{1-4\alpha}{2\alpha}\cdot 2\alpha N+ 3\alpha N = (1-\alpha)N. 
	\]
	Hence, the set of non-neighbors of $X_2$ in $G$ has size at least $\alpha N$, contradicting the assumption that $G$ is $\alpha$-joined. 
	Thus indeed $|X_1|,\,|X_2|\leq \alpha N$. 
	
	Let $G'$ be the graph obtained from $G$ by removing the vertices in $X_1\cup X_2$. That is $G'=G'(V_1', V_2')$ with $V_1'=V_1\setminus X_1$ and $V_2'=V_2\setminus X_2$.
	It is clear that for each subset $X\subseteq V'_i$ with $|X|\leq \alpha N$ we have
	$|N_{G'}(X)|> (1-4\alpha)/{(2\alpha)} |X|$. If $|X|> \alpha N$, then since $G$ is $\alpha$-joined, the set of non-neighbors of $X$ in $G$ has size less than $\alpha N$. Thus
	\[
	|N_{G'}(X)|\geq |V_{i+1}'|-\alpha N \geq (1-2\alpha)N.
	\] 
	
	To verify that $G'$ is $(6\alpha N, (1-2\alpha)/(6\alpha))$-expanding, note that if $|X|\leq \alpha N$, then 
	\[
	|N_{G'}(X)|> \dfrac{1-4\alpha}{2\alpha} |X| \geq \dfrac{1-2\alpha}{6\alpha}|X|,
	\]
	while if $\alpha N< |X|\leq 6\alpha N$, then 
	\[
	|N_{G'}(X)|\geq (1-2\alpha) N \geq \dfrac{1-2\alpha}{6\alpha}|X|.
	\]
	
Finally, let $F=F(U_1, U_2)$ be the null graph with $\alpha N$ vertices in each part. Set $Y_1'=Y_1\setminus X_1$ and $Y_2'=Y_2\setminus X_2$. Clearly, $|Y_1'|, \,|Y_2'|\geq \alpha N$. Let $\phi:F\hookrightarrow G'$ be an embedding such that maps all vertices in $U_i$ to $Y_i'$ for $i\in \{1, 2\}$. We prove that $\phi$ is $ (6\alpha N, (1-4\alpha)/(6\alpha))$-good embedding. To see this, let $X\subseteq V_i'$ be fixed for $i\in \{1, 2\}$. If  $|X|\leq \alpha N$, then 
	\[|N_{G'}(X)\setminus (Y_{i+1}')|> \dfrac{1-4\alpha}{2\alpha} |X|-\alpha N \geq \dfrac{1-6\alpha}{2\alpha} |X|\geq \dfrac{1-4\alpha}{6\alpha} |X|+ |X\cap Y_i'|.  \]
That is 
\[
|N_{G'}(X)\setminus\phi(F)|\geq \frac{1-6\alpha}{2\alpha}|X|+|X\cap \phi(F)|.
\]
	If $\alpha N< |X|\leq 6\alpha N$, then as $ |X\cap \phi(F)|\leq \alpha N $ we have
	\[
	|N_{G'}(X)\setminus F|\geq (1-2\alpha)N-\alpha N\geq \dfrac{1-4\alpha}{6\alpha} |X|+|X\cap \phi(F)|.  
	\]
	This shows that $\phi:F\hookrightarrow G'$ is a $(6\alpha N, {(1-4\alpha)}/{(6\alpha)})$-good embedding. 
\end{proof}

\section{Proof of Theorem~\ref{thm:main}}
In this section, we prove Theorem~\ref{thm:main}. Our approach proceeds in two steps.  
First, we show that for a suitable choice of $\alpha$, every $\alpha$-joined bipartite graph contains the long bipartite subdivision of any bounded degree graph (Lemma~\ref{lem:alphajoined}).  
Second, we prove that if $G$ is the random bipartite graph provided by Lemma~\ref{quasi}, then every sufficiently large subgraph $J \subseteq G$ with $e(J)\geq e(G)/r$ (with $r\ge 2$) is $\alpha$-joined, and therefore $J$ contains the long bipartite subdivision of every bounded degree graph as a subgraph.

\begin{lemma} \label{lem:alphajoined}
	Let $n$ and $D$ be positive integers.
	Suppose that $H$ is a graph and $H^\sigma$ is an $(n,D)$-bipartite subdivision of $H$. 
	Also, let $G=G(V_1,V_2)$ be an $\alpha$-joined bipartite graph, where $\alpha\leq 1/(6D+14)$  and  $|V_i|=N\geq n/\alpha$ for each $i\in\{1,2\}$. If  $\sigma(e)\geq 2\log_{D-1} (\alpha N)$ for every edge $e\in E(H)$, then $G$ contains a copy of $H^\sigma$ as a subgraph.   
\end{lemma}

\begin{proof}
	Suppose that $(A_1,A_2)$ is a bipartition of $H^\sigma$ and let $|V(H)\cap A_1|=r_1$ and $|V(H)\cap A_2|=r_2$. By Lemma~\ref{lem:expender_in_join}, there exists an induced subgraph $G'=G'(V'_1,V'_2)$ of $G$ which is $(6\alpha N,D+2)$-expanding since $D+2\leq (1-2\alpha)/(6\alpha)$. Also, there exists  a $(6\alpha N, (1-2\alpha)/(6\alpha))$-good embedding from a null graph $F$ into $G'$, where $F$ has $r_i$ vertices in $V_i$, i.e. we can embed vertices of $H$ into $G'$ as a good embedding. This can be done because $r_i\leq n\leq \alpha N$.
	Now, consider an arbitrary ordering on the edge set $E(H)$, say $e_1,\ldots, e_m$. 
	For each $ 0 \leq i \leq m $,
	let $H_i$ be the spanning subgraph of $H$ induced on the edges $e_1,\ldots, e_i$. 
	Note that $ H_0 $ is just a null graph on the vertex set $ V(H) $.
	We inductively embed $H_i^\sigma$ into $G'$ for each $ 0\leq i\leq m $, such that the embedding $\phi_i: H_i^\sigma \hookrightarrow G'$ is $(6\alpha N,D)$-good. Since $H_0^\sigma=F$ and $D\leq  (1-2\alpha)/(6\alpha)$, we already have a $(6\alpha N,D)$-good embedding $\phi_0: H_0^\sigma\hookrightarrow G'$.  
	
	Now, suppose that we have already embedded $H_{i-1}^\sigma$ and let $e_i=v_1v_2$. We want to embed a path of length $\sigma(e_i)$ between $v_1$ and $v_2$. If $\sigma(e_i)$ is even, then $v_1$ and $v_2$ are in the same part of $G'$ and in case $\sigma(e_i)$ is odd, $v_1$ and $v_2$ are in different parts of $G'$. 
	
	First suppose that $\sigma(e_i)$ is odd. Consider two isomorphic trees $T_1$ and $T_2$ rooted at $v_1$ and $v_2$, where $T_j$ consists of a path of length $(\sigma(e_i)-2k-1)/2$ starting from $v_j$  appending to a tree of maximum degree $D$ with $\lceil \alpha N\rceil $ leaves and height $k= \lceil\log_{D-1} (\alpha N)\rceil$, for every $j\in \{1, 2\}$.
	
	Now, we are going to apply Theorem~\ref{thm:good_adding} to extend embedding of $H^{\sigma}_{i-1}$ to $H^{\sigma}_{i-1}\cup T_1\cup T_2$. The former bipartite graph has at most $n+2\alpha N$ vertices in each part and by the choice $N$ is a $(3\alpha N, D)$-bipartite graph. Therefore, $\phi_{i-1}$ can be extended to a $(6\alpha N, D )$-good embedding of $H^{\sigma}_{i-1}\cup T_1\cup T_2$. Since $v_1$ and $v_2$ are in different parts in $G'$ and the heights of $T_1$ and $T_2$ are the same, the sets of 	
	leaves of $T_1$ and $T_2$ are in different parts in $G'$. As $G$ is $\alpha$-joined, there exists an edge $f$ between a leaf of $T_1$, say $x_1$, and a leaf of $T_2$, say $x_2$. This completes the path $P_i$ of length $\sigma(e_i)$ which starts from $v_1$, passes through $T_1$ to $x_1$ and then continues from $f$ to $x_2$ and then from  $x_2$ to $v_2$ via $T_2$.
	
	Finally, using Lemma~\ref{lem:good_removing}, we remove vertices of $(V(T_1)\cup V(T_2))\setminus V(P_i)$ to obtain an embedding $\phi_i: H^\sigma_i \hookrightarrow G'$ which is $(6\alpha N, D)$-good.  
	
	In the case that $\sigma(e_i)$ is even, we use the same argument to embed the path of length $\sigma(e_i)$ between $v_1$ and $v_2$. The only difference is that the trees $T_1$ and $T_2$ are rooted at $v_1$ and $v_2$, where $T_j$ consists of a path of length $\sigma(e_i)/2-k-2+j$ starting from $v_j$  appending to a tree of maximum degree $D$ with $\lceil \alpha N\rceil $ leaves and height $k= \lceil\log_{D-1} (\alpha N)\rceil$. Note that in this case, the parity of the heights of $T_1$ and $T_2$ are different and so the set of leaves of $T_1$ and $T_2$ will be in different parts of $G'$. Therefore, we can use the $\alpha$-jointness of $G$ to connect the set of leaves of $T_1$ and $T_2$ with an edge.  
\end{proof}

Now, we are ready to prove Theorem \ref{thm:main}.
\begin{proof}
	Let $\alpha={1}/{(6D+14)}$, $c_1=(6D+14) r^{-6\log \alpha/\alpha}$, $c_2=4D^2\log(6D+14)r^2\log^2 r$ and $c_3=6D+14$. Also, let $N=c_1n$ and $p=c_2/n$ and set $\delta=\sqrt{\frac{7\log (c_1/c_3)}{c_2c_3}}$. One can check that $  3/2\geq \delta >\sqrt{6{\log({c_1e}/{c_3})/(c_2c_3)}}$. Therefore,
	by Lemma~\ref{quasi} there exists a bipartite graph $G=G(V_1,V_2)$ such that $|V_1|=|V_2|=N$ and
	\begin{enumerate}
		\item $e(G)= (1\pm n^{-1/4}) pN^2$. In particular, $e(G)\leq r^{400D\log D}n$.
		\item For every two subsets $U\subseteq V_1$ and $W\subseteq V_2$, with $|U|=u$ and $|W|=w$, where $u,\, w\geq c_3n$, we have 
		\[
		e_G(U, W)= (1\pm \delta) puw.
		\]    
	\end{enumerate}
	
	Let $G'$ be a spanning subgraph of $G$ where $e(G')\geq 1/r e(G)$ and let $F$ be the graph obtained from $G$ by removing all edges of $G'$. Suppose that $G'$ does not contain any copy of $H^\sigma$ as a subgraph. We claim that for every sets $V_1'\subseteq V_1$ and $V'_2\subseteq V_2$ with $|V'_1|=|V'_2|=N'$, we have 
	\begin{align*}\label{eq:main}
e_F(V'_1,V'_2)\geq (1-\delta)\left(1-\left(\dfrac{n}{\alpha N'}\right)^\lambda\right) pN'^2,
	\end{align*}
	where $\lambda=-\alpha/(3\log \alpha)$.
	
	We prove this claim by induction on $N'$. Let $V_i'\subseteq V_i$ for each $i\in \{1, 2\}$ such that $|V_1'|=|V_2'|=N'$.
	If $N'< n/\alpha$, then the right hand side is negative and we are done. So, suppose that $N'\geq n/\alpha$. Let $G''$ be the induced subgraph of $G'$ on $V'_1\cup V'_2$. Since $G''$ does not contain $H^\sigma$,  Lemma~\ref{lem:alphajoined} implies that $G''$ is not $\alpha$-joined. So, there are two subsets $A^1_1\subseteq V'_1$, $A^1_2\subseteq V'_2$ with $|A^1_1|=|A^1_2|=\alpha N'$, such that $e_{G''}(A^1_1,A^1_2)=0$. Now, for each $i\in \{1,2\}$ partition $V'_i\setminus A^1_i$, into $1/\alpha-1$ subsets $A^2_i,\ldots, A^{1/\alpha}_i$ each of size $\alpha N'$. Also, let $B_i=\cup_{j=2}^{1/\alpha} A^j_i$. Therefore, by induction hypothesis, we have
	
	\begin{align*}
	e_F(V_1',V_2')&= e_G(A^1_1,A^1_2)+ e_F(B_1,B_2)+ \sum_{i=2}^{1/\alpha} e_F(A^1_1,A^i_2) +\sum_{i=2}^{1/\alpha} e_F(A^i_1,A^1_2) \\
	&\geq (1-\delta)pN'^2\\
	&~~~
	\left[ \alpha^2 + \left(1-\left(\dfrac{n}{\alpha(1-\alpha)N'}\right)^\lambda\right) (1-\alpha)^2+  (2/\alpha-2)\left(1-\left(\dfrac{n}{\alpha^2 N'}\right)^\lambda\right) \alpha^2 \right] \\
	&= (1-\delta)pN'^2\\
	&~~~
	\left[  \alpha^2 + (1-\alpha)^2 - \left(\dfrac{n}{\alpha(1-\alpha)N'}\right)^\lambda (1-\alpha)^2 + (2\alpha-2\alpha^2)- \left(\dfrac{n}{\alpha^2 N'}\right)^\lambda(2\alpha-2\alpha^2) \right] \\
	&=  (1-\delta)pN'^2 \left[ 1- \left(\dfrac{n}{\alpha N'}\right)^\lambda\left( (1-\alpha)^{2-\lambda} +2\alpha^{1-\lambda} -2\alpha^{2-\lambda}\right) \right].
	\end{align*}
	
	Consider the function 
	\[
	f(\alpha)= (1-\alpha)^{2-\lambda}+2\alpha^{1-\lambda}- 2 \alpha^{2-\lambda} = 
	(1-\alpha)^{2-\lambda}+2\alpha e^{\alpha/3}- 2 \alpha^{2} e^{\alpha/3}.
	\]
	 If we write the Taylor's expansion of $f$ and since $\alpha$ is small by the growth of $D$, we have 
	\begin{align*}
	f(\alpha)&= 1-(2-\lambda) \alpha+ (2-\lambda)(1-\lambda) \dfrac{\alpha^2}{2}+ 2\alpha+2\alpha^2/3-2\alpha^2 +O(\alpha^3)\\
	&= 1-\dfrac{\alpha^2}{3}-\dfrac{\alpha^2}{3\log \alpha}+o(\alpha^2)<1.
	\end{align*}
	Hence, 
	\[e_F(V_1',V_2') \geq   (1-\delta)pN'^2 \left[ 1- \left(\dfrac{n}{\alpha N'}\right)^\lambda \right], \]
	as desired. 
	Therefore,
	
	\begin{align*}
	(1-\delta)pN^2 \left[ 1- \left(\dfrac{n}{\alpha N}\right)^\lambda \right] \leq e(F) \leq \left(1-\dfrac{1}{r}\right) e(G)\leq \left(1-\dfrac{1}{r}\right) (1+n^{-1/4}) pN^2. 
	\end{align*}
	Thus,
	\[
	(1-\delta) \left[ 1- \left(\dfrac{n}{\alpha N}\right)^\lambda \right] \leq \left(1-\dfrac{1}{r}\right) (1+n^{-1/4}).
	\]
	On the other hand by the choice of $c_1,\, c_2$ and $c_3$ we have 
	\[(1-\delta) \left[ 1- \left(\dfrac{n}{\alpha N}\right)^\lambda \right]\geq  \left(1- \dfrac{1}{r\log r}\right) \left(1-\dfrac{1}{r^2}\right),\]
	which is a contradiction.
	
	This contradiction shows that there should be a copy of $H^{\sigma}$ in $G'$ implying 
	that~$G\longrightarrow(H^{\sigma})_r$. Also, by Lemma \ref{quasi} , we have 
	\[
	\widehat{R}_r(H)\leq |E(G)|\leq (1+n^{-1/4})pN^2 \leq r^{400D\log D}n,
      \]
      as required
\end{proof}

\begin{remark}
  Lemma~\ref{lem:alphajoined} can be viewed as a universality statement. The $\alpha$-joined bipartite graphs under consideration contain every bipartite subdivision of the type described. In particular, the proof of Theorem~\ref{thm:main} shows that in any edge-coloring of the bipartite graph under study, one can always find a monochromatic subgraph that contains all such bipartite subdivisions. 
  Following \cite{KRSSZ}, a graph $G$ is said to be
  $r$-\textit{partition universal} for a class of graphs
  $\mathcal{H}$ if, under every $r$-edge coloring of $G$, there
  exists a color class that contains copies of all graphs in~$\mathcal{H}$.  Theorem~\ref{thm:main} establishes that the bipartite graph we color is $k$-partition universal for the class of all bipartite subdivisions described above.
\end{remark}


\providecommand{\bysame}{\leavevmode\hbox to3em{\hrulefill}\thinspace}
\providecommand{\MR}{\relax\ifhmode\unskip\space\fi MR }
\providecommand{\arXiv}{\relax\ifhmode\unskip\space\fi arXiv }
\providecommand{\MRhref}[2]{%
  \href{http://www.ams.org/mathscinet-getitem?mr=#1}{#2}
}
\renewcommand{\MR}[1]{%
  \href{http://www.ams.org/mathscinet-getitem?mr=#1}{MR~#1}
}
\renewcommand{\arXiv}[1]{%
  Available as \href{https://arxiv.org/abs/#1}{arXiv:#1}.
}
\providecommand{\href}[2]{#2}

\appendix
\normalsize
\renewcommand\thetheorem{\Roman{theorem}}
\section{Proof of Theorem~\ref{thm:good_adding}} \label{app}
\begin{proof}
	The argument parallels the strategy of Theorem~2.3 in~\cite{Rolling}, suitably adapted to the bipartite setting.
	
	We proceed by induction on the number of pendant vertices added.  
	It is enough to treat the case where $F'$ is obtained from $F$ by adjoining a single new leaf $v$ to an existing vertex $w\in V(F)$; the general statement then follows by repeated application.
	
	Let $G=G(V_1,V_2)$ be as in the statement, and for $X\subseteq V_i$ and an embedding $f:H\hookrightarrow G$, define
	\[
	R(X,f) \;=\; \big|N_G(X)\setminus f(H)\big| \;-\; \sum_{x\in X} \big(D - \deg_H(f^{-1}(x))\big) \;-\; |f(H)\cap X|.
	\]
	Here, $\deg_H(f^{-1}(x))=0$ if $x$ is not in the image of $f$.
	
	Let $Y := N_G(\phi(w))\setminus \phi(F)$.  
	For each $a\in Y$, we can extend $\phi$ to an embedding $\phi_a:F'\hookrightarrow G$ by mapping $v\mapsto a$ while keeping $\phi_a$ identical to $\phi$ on $F$.
	
	Our goal is to find some $a\in Y$ for which
	\[
	R(X,\phi_a) \ge 0
	\quad\text{for all } X\subseteq V_i, \; i\in\{1,2\}, \; |X|\le 2n.
	\]
	Assume for contradiction that no such $a$ exists.  
	Then, for each $a\in Y$, there is some $X_a\subseteq V_i$ with $|X_a|\le 2n$ and $R(X_a,\phi_a) < 0$.  
	Since $\phi$ is $(2n,D)$-good, we have $R(X,\phi)\ge 0$ for all subsets $X$ with $|X|\le 2n$.
	
	Note that 
	\begin{align*}
	R(X_a,\phi_a) - R(X_a,\phi) 
	&= -\mathbb{1}[a \in N_G(X_a)] + \mathbb{1}[a\in X_a] + \mathbb{1}[\phi(w)\in X_a] - \mathbb{1}[a\in X_a]\\
	&= \mathbb{1}[\phi(w)\in X_a] - \mathbb{1}[a \in N_G(X_a)].
	\end{align*}
	Thus, $R(X_a,\phi_a) < 0$ implies $R(X_a,\phi) = 0$, $\phi(w)\notin X_a$, and $a\in N_G(X_a)$.
	
	We require three auxiliary observations.
	
	\begin{claim}\label{claim:smallX}
		If $ X \subseteq V_i $ for $i\in \{1, 2\}$ is such that $ R(X,\phi) = 0 $ and $ |X|\leq 2n $ then $ |X|\leq n $.
	\end{claim}
	\begin{proof}
		Since $G$ is $(2n,D+2)$-expanding, $|N_G(X)| \ge (D+2)|X|$.  
		Also, in each part $V_i$, the embedding $\phi(F)$ occupies at most $n$ vertices. Hence
		\[
		0 = R(X,\phi) \ge (D+2)|X| - n - D|X| - |X| = |X| - n,
		\]
		yielding $|X|\le n$.
	\end{proof}
	
	\begin{claim}\label{claim:submod}
		The map $X\mapsto R(X,\phi)$ is submodular:
		\[
		R(A\cup B,\phi) + R(A\cap B,\phi) \;\le\; R(A,\phi) + R(B,\phi)
		\quad\text{for all }A,B\subseteq V(G).
		\]
	\end{claim}
	\begin{proof}
		The first term of $ R(X,\phi) $ is a submodular function of $ X $, and the other two are modular,
		so the claim follows.
	\end{proof} 
	\begin{claim}\label{claim:union}
		If for $ A, B\subseteq V(G) $ it holds that $ R(A, \phi) = R(B, \phi) = 0 $, and $ |A|$, $|B| \le n $,
		then $ R(A\cup B, \phi) = 0 $ and $ |A\cup B|\le n$.
	\end{claim} 
	\begin{proof}
		Since $\phi$ is $(n,D)$-good and $|A\cup B|, |A\cap B|\le 2n$, both $R(A\cup B,\phi)$ and $R(A\cap B,\phi)$ are non-negative.  
		By Claim~\ref{claim:submod} we get $R(A\cup B,\phi)=0$, and then Claim~\ref{claim:smallX} yields $|A\cup B|\le n$.
	\end{proof}
	
	Returning to the main argument, for each $a\in Y$ we have $R(X_a,\phi)=0$ and $|X_a|\le n$.  
	Let $X^* := \bigcup_{a\in Y} X_a$.  
	If $Y\neq\emptyset$, repeated application of Claim~\ref{claim:union} gives $R(X^*,\phi)=0$ and $|X^*|\le n$.
	Now set $X' := X^* \cup \{\phi(w)\}$.  
	Since $\phi(w)\notin X_a$ for any $a$, we have $\phi(w)\notin X^*$ but $Y\subseteq N_G(X^*)$.  
	Therefore, the neighborhood term in $R(X^*,\phi)$ equals that in $R(X',\phi)$, while the degree-deficit term strictly increases by at least $1$ when passing from $X^*$ to $X'$ (because $\deg_F(w)\le D-1$).  
	That means
	\[
	R(X', \phi) \leq R(X^*, \phi)- \big(D -\deg_F(w)\big) \leq  R(X^*, \phi)-1 < 0.
	\]  
	contradicting $ R(X', \phi) \ge 0 $.
	
	Thus, such an $a\in Y$ exists, completing the inductive step and the proof
\end{proof}
\section{Proof of Lemma~\ref{lem:good_removing}}\label{app2}.
Here we give a proof of Lemma~\ref{lem:good_removing}. The argument follows the same  as Lemma~2.4 in~\cite{Rolling}.
It suffices to verify the statement when $F'$ is obtained from $F$ by removing exactly one vertex $v$ of degree~$1$; the general case then follows by iterating the argument.

Let $ \phi' $ be a restriction of $ F $ to such $ F' $, and let
$ w \in F' $ denote the unique neighbor of $ v $.
Suppose that $ X\subseteq V_i $ is an arbitrary subset of vertices with $|X|\leq n$.

We distinguish the following two cases. 

\smallskip
\noindent
\emph{Case 1:} $\phi(v) \in X$. 
Since $v$ and $w$ are adjacent in $F$ and $G$ is bipartite, we have $\phi(w) \notin X$. Deleting $v$ leaves the left-hand side of~\eqref{good-embedding} unchanged. On the right-hand side, the degree term $\sum_{v\in X} (D - \deg_F(\phi^{-1}(v)))$ increases by $1$, while $|\phi(F) \cap X|$ decreases by $1$. These changes exactly offset, so inequality~\eqref{good-embedding} still holds for $\phi'$.

\smallskip
\noindent
\emph{Case 2:} $\phi(v) \notin X$.  
If $\phi(w) \notin X$, then removing $v$ has no effect on either side of~\eqref{good-embedding}, and the claim follows immediately.  
If instead $\phi(w) \in X$, then necessarily $\phi(w) \in V_i$ and $\phi(v) \in V_{i+1}$. Removing $v$ adds one new vertex to $N_G(X) \setminus \phi(F)$, increasing the left-hand side by~$1$. On the right-hand side, the degree term increases by~$1$ (since $\phi(v) \in N_G(\phi(w))$), while $|\phi(F) \cap X|$ remains unchanged. Thus~\eqref{good-embedding} continues to hold for $\phi'$.
\smallskip
In all cases, the $(n,D)$-good property in Inequality \ref{good-embedding} is preserved when a degree-$1$ vertex is removed.

\endgroup
\end{document}